\documentclass[psamsfonts]{amsart}  

\usepackage{amssymb,amsfonts}
\usepackage[all,arc]{xy}
\usepackage{mathrsfs}

\usepackage{tikz}
\usepackage{ctable}
\usepackage{tabulary}
\newtheorem{thm}{Theorem}

\newtheorem{lem}{Lemma}
\newtheorem{conj}{Conjecture}

\theoremstyle{definition}
\newtheorem{defn}[thm]{Definition}

\theoremstyle{remark}

\newcommand{\Log}{\mathop{\mathrm{Log}}\nolimits}
\newcommand{\Arg}{\mathop{\mathrm{Arg}}\nolimits}

\renewcommand{\Re}{\mathop{\mathrm{Re}}\nolimits}

\newcommand{\CC}{\mathbb{C}}
\newcommand{\RR}{\mathbb{R}}

\title{Toric Cycles in the~Complement of a Complex Curve in~$(\mathbb{C}^{\times})^2$}

\author{Alexey Lushin, Dmitry Pochekutov}

\begin{document}

\begin{abstract}
The amoeba of a~complex curve in the 2-dimensional complex
torus is its image under the~projection onto the~real subspace
in the~logarithmic scale. The complement to an~amoeba is a~disjoint 
union of connected components that are open and convex. A~toric cycle
is a~2-cycle in the~complement to a~curve associated with a~component
of the~complement to an~amoeba. We prove homological independence
of toric cycles in the~complement to  a~complex algebraic curve with amoeba of maximal area.

\end{abstract}

\maketitle

\section{Introduction}
\label{s:0}
It is difficult to overestimate the~importance of homological characteristics of algebraic hypersurfaces
and their complements to the~complex space. For instance, constructing of the~dual bases of the~homology
and the~cohomology for a~complement of an~algebraic set plays a~crucial role in the~multidimensional residue theory. To trace  the~history of these problems see \cite{Pncre, Leray, TsYg} and \cite[Sect.~13]{AizYu}.

Certain information on homological cycles in the~complement of an~algebraic set can be inferred from studying its amoeba and coamoeba. 

Given an algebraic hypersurface $V=P^{-1}(0)\cap (\mathbb{C}^{\times})^n$ defined as a zero locus in the~complex torus $(\mathbb{C}^{\times})^n=(\mathbb{C}\setminus\{0\})^n$ of a polynomial $P : \mathbb{C}^n\to\CC$, consider the~\textit{amoeba} $\mathscr{A}_V$ of $V$ (or $\mathscr{A}_P$ of $P$), i.e. the~image of $V$  under the~logarithmic mapping $$\Log(z)=(\log |z_1|,\ldots,\log |z_n|).$$
The~complement $\mathbb{R}^n\setminus\mathscr{A}_V$ consists of a finite number of connected components~$E_i$ for $i=1,\ldots, s$.
Each component $E_i$  corresponds to an integer point~$\nu\in\Delta_P$. So we denote by $E_{\nu}$ the~component $E_i$ 
(see Section~\ref{s:1} for details).

Let $x\in E_{\nu}$. Then we call an $n$-dimensional real torus $$\Gamma_{\nu}(x)=\Log^{-1}(x)$$ \textit{a~toric cycle} in $(\mathbb{C}^{\times})^n \setminus V$. We drop~$x$ in the~notation of $\Gamma_{\nu}$ since cycles $\Gamma_{\nu}(x)$ and $\Gamma_{\nu}(y)$ are homologically equivalent for $x,y\in E_{\nu}$.
When $\nu$ is a vertex of $\Delta_P$, A.G~Khovanskii and O.A.~Gelfond called $\Gamma_\nu$
the~cycle related to a vertex of the~Newton polytope. The~sum of Grothendieck residues associated to a polynomial
mapping $(P_1,\ldots, P_n):(\mathbb{C}^{\times})^n\to \mathbb{C}^n$ can be represented in~terms of such cycles for
the~hypersurface~$P=P_1\cdot\ldots \cdot P_n$~\cite{GeKh}. M.A.~Mkrtchan and A.P.~Yuzhakov in \cite{MkrtYu} proved that cycles $\Gamma_\nu$ related to vertices of $\Delta_P$  are homologically independent in the~group $\textup{H}_n((\mathbb{C}^{\times})^n \setminus V)$).

The following natural conjecture has arisen  in the~context of works \cite{Forsberg, FPT} of M.~Forsberg, M.~Passare and A.~Tsikh on amoebas of algebraic hypersurfaces. Explicitly it was stated in \cite{BuTs} as following.

\begin{conj}
	The toric cycles $\Gamma_{\nu}$ constitute a homologically independent family in the~homology group $H_n((\mathbb{C}^{\times})^n \setminus V).$
\end{conj}

In~\cite{FPT} it was proved that the~family of cycles $\Gamma_\nu$ is a basis for the~homology group $\textup{H}_n((\mathbb{C}^{\times})^n \setminus V)$), when $V$ is a hyperplane arrangement in so-called optimal position.

In this paper we focus on the~bivariate case $n=2$. The amoeba of a~complex algebraic curve in $(\mathbb{C}^{\times})^2$
has finite area bounded from above by an~expression in terms of the~degree of the~curve \cite{RuPa}.
A bivariate polynomial~$P(z,w)$ is called \textit{Harnack} if the~amoeba of~$P$ has the~maximal area \cite{Pa}.
Complex curves defined by Harnack polynomials compose an important class since their real parts are isotopic
to Harnack curves in $(\mathbb{R}^{\times})^2=(\mathbb{R}\setminus\{0\})^2$, which arise in topics related to the~Hilbert sixteenth problem (cf. \cite{Mi}).

The main result of the~present paper is
\begin{thm}
	Let $V$ be an algebraic complex curve in $(\mathbb{C}^{\times})^2$ defined by a Harnack polynomial~$P$.
	Then the~toric cycles $\Gamma_\nu$ constitute a homologically independent family in the~homology group $H_2((\mathbb{C}^{\times})^2 \setminus V).$
\end{thm}

The proof of this theorem (see Section~\ref{s:3}) is based on trigonometric properties of complex algebraic curves defined 
by Harnack polynomials (see Section~\ref{s:2}) and their projections (see Section~\ref{s:1}); it employs basic techniques of algebraic
topology.

\section{Amoebas and coamoebas of algebraic hypersurfaces}
\label{s:1}

Denote $\mathbb{C}^{\times}=\mathbb{C}\setminus\{0\}$.
Let $V=P^{-1}(0)\cap(\mathbb{C}^{\times})^n$ be an algebraic hypersurface, where $P : \mathbb{C}^n\to\mathbb{C}$ is a polynomial.
Its \textit{amoeba} $\mathscr{A}_V$ (or the~amoeba $\mathscr{A}_P$ of P) is the~image  of $V$ under the~logarithmic mapping $\Log: (\CC^\times)^n \to \RR^n$ given by the~formula
\begin{equation*}
	\Log: (z_1,\ldots,z_n) \mapsto (\log |z_1|,\ldots,\log |z_n|).
\end{equation*}

Similarly the~\textit{coamoeba} $\mathscr{A}'_V$ (or the~coamoeba $\mathscr{A}'_P$ of P) is the~image of $V$ under the~argument projection $\Arg: (\CC^\times)^n \to (-\pi, \pi]^n$ given by the~formula
\begin{equation*}
	\Arg: (z_1,\ldots,z_n)\mapsto (\arg z_1,\ldots,\arg z_n).
\end{equation*}

Fig.~1 depicts the~amoeba~$\mathscr{A}_P$ and the~coamoeba~$\mathscr{A}'_P$
for $P(z,w)=z^2w-4zw+zw^2+1$.

\begin{figure}[ht]
\label{fig:1}
\begin{center}
\includegraphics[scale=.53]{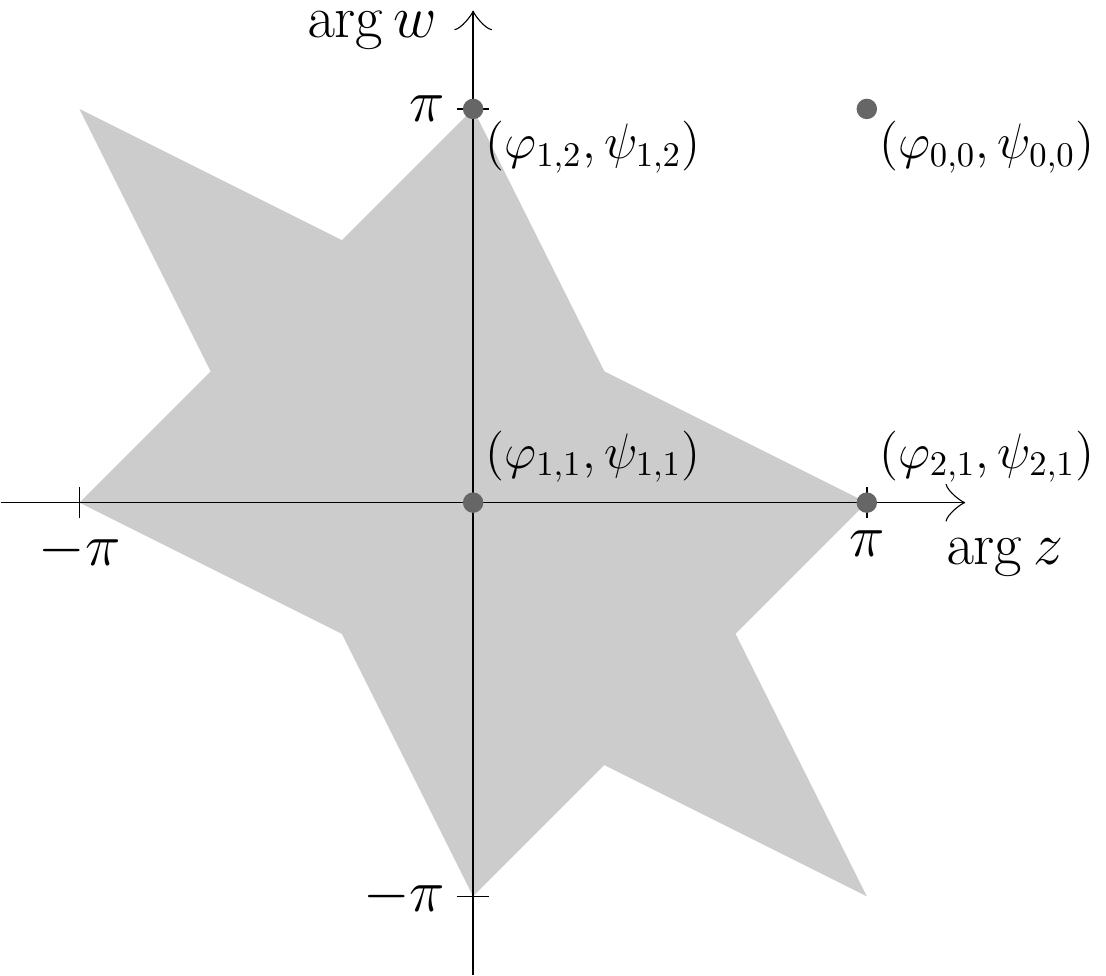}
\includegraphics[scale=.49]{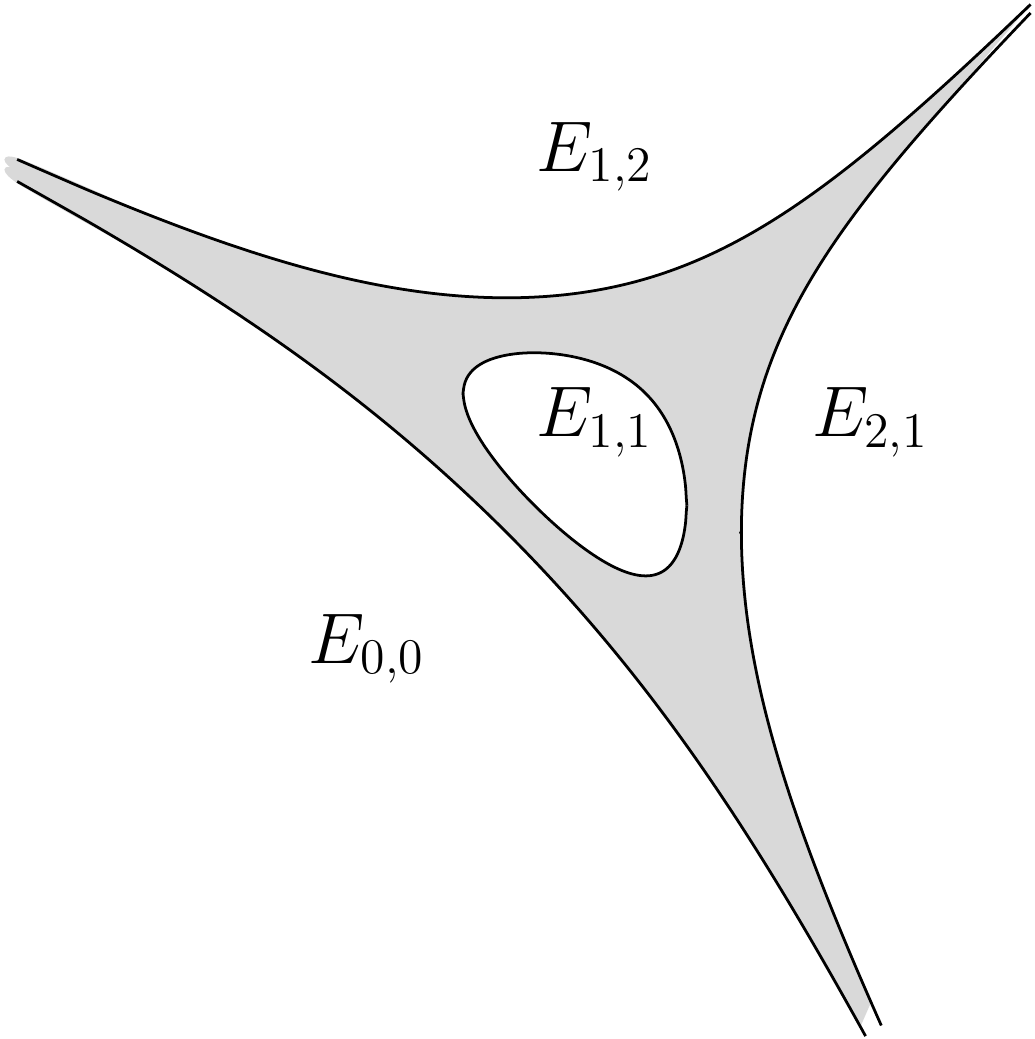}
\caption{\small  The coamoeba (left) 
and the~amoeba with its complement components $E_\nu$ (right) for the~polynomial
$P(z,w)=z^2 w - 4zw + zw^2 + 1$.}
\end{center}
\end{figure}

Since $\mathscr{A}_V$ is a closed subset of $\mathbb{R}^n$, the~complement $\mathbb{R}^n\setminus\mathscr{A}_V$ is open.
It consists of a finite number of connected components $E_i$, which are convex \cite[Section~6.1]{GKZ}.
The structure of $\mathbb{R}^n\setminus\mathscr{A}_V$ can be read from the~Newton polytope $\Delta_P$ of a polynomial~$P,$ i.e. the~convex hull in $\mathbb{R}^n$ of the~list of exponents of terms present in $P$.
Recall that the~dual cone at a point $\nu\in\Delta_P$ to $\Delta_P$ is
\begin{equation*}
	C^\vee_\nu(\Delta_P) = \{ s\in\RR^n : \langle s,\nu \rangle = \max_{\alpha\in\Delta_P} \langle s,\alpha \rangle \}.
\end{equation*}
A recession cone~$C(E)$ of the~convex set~$E$ is the~maximal cone that can be put inside~$E$ by a translation.

The following theorem is a summary of Propositions~2.4, 2.5 and 2.6 in~\cite{FPT}.

\begin{thm}
	There exists an injective order mapping
		\begin{equation*}
		\textup{ord}: \{E\}\to \Delta_f\cap \mathbb{Z}^n
		\end{equation*}
	such that the~dual cone $C^\vee_{\nu}(\Delta_P),$ $\nu=\textup{ord}(E)$, is the~recession cone of $E$.
\end{thm}
So the~theorem states that one can encode a component $E$ of the~complement $\mathbb{R}^n\setminus\mathscr{A}_V$ as $E_\nu$,
where $\nu\in\Delta_P\cap \mathbb{Z}^n$ and $\nu=\textup{ord}(E).$
We refer readers to Fig.~1 to observe this correspondence for the~hypersurface defined by $P(z,w)=z^2 w - 4zw + zw^2 + 1$. Its Newton polytope~$\Delta_P$ is the~convex hull
in $\mathbb{R}^2$ of points $(0,0)$, $(1,2)$, $(1,1)$ and $(2,1)$. 

The order mapping $\textup{ord}$ can be defined in terms of the~\textit{Ronkin} 
\textit{function}~$\mathcal{N}_P:\mathbb{R}^n\to\mathbb{R}$, that is the~mean value integral
\begin{equation}
	\mathcal{N}_P(x) = \frac{1}{(2\pi i)^n} \int\limits_{\Log^{-1}(x)} 
	\log\big|P(z_1,\ldots,z_n)\big| 
	\frac{dz_1\wedge\ldots\wedge dz_n}{z_1\cdot\ldots\cdot z_n},
\end{equation}
where $x=(x_1,\ldots, x_n)\in \mathbb{R}^n$. It is affine linear in a~component~$E$ of the~complement~$\mathbb{R}^n\setminus\mathscr{A}_V$.
Moreover, the~gradient $ \left. \textup{grad}\,\mathcal{N}_P \right|_E$ is equal to $\textup{ord}(E)$ \cite{Ronkin}.

We define \textit{a~toric cycle} in $(\mathbb{C}^{\times})^n\setminus V$
related to $E_\nu$ to be an $n$-dimensional real torus
\begin{equation*}
	\Gamma_\nu=\textup{Log}^{-1}(x)=\{z:|z_1|=e^{x_1},\ldots, |z_n|=e^{x_n}\},
	\end{equation*}
$x=(x_1,\ldots, x_n)\in E_\nu$. Since $E_{\nu}$ is convex, it contains the segment $[x;y]$ for any pair~$x,y\in E_\nu$.
It follows $\Gamma_{\nu}(y)-\Gamma_{\nu}(x)=\partial\,\textup{Log}^{-1}[x;y]$, and so $\Gamma_{\nu}(x)$ 
is homologically equivalent to $\Gamma_{\nu}(y)$, therefore we drop $x$ in notation of~$\Gamma_{\nu}$.

\section{Harnack polynomials and their amoebas}
\label{s:2}
An amoeba is a closed but non-compact subset of $\mathbb{R}^n$. Nevertheless, amoebas in $\mathbb{R}^2$
have finite areas. Moreover, M.~Passare and H.~Rullg\aa rd showed in~\cite{RuPa} that
\begin{equation*}
	Area(\mathscr{A}_P)\leq \pi^2 Area(\Delta_P)
\end{equation*}
for a bivariate polynomial~$P(z,w)$.

\begin{defn}[\cite{Pa}]
A polynomial $P:\mathbb{C}^2\to \mathbb{C}$ is called \textit{Harnack} if its Newton polygon~$\Delta_P$ has a non-zero
area, and the~area of its amoeba is maximal, i.e.
\begin{equation}
\label{eq:areas}
	Area(\mathscr{A}_P)=\pi^2 Area(\Delta_P).
\end{equation}
\end{defn}
Note that G.~Mikhalkin showed in~\cite{Mi} that for a given lattice polygon $\Delta$ one can construct a polynomial~$P(z,w)$
with $\Delta_P=\Delta$, such that equality~(\ref{eq:areas}) holds. The amoeba depicted in Fig.~1 belongs to the~Harnack
polynomial $P(z,w)=z^2w-4zw+zw^2+1$.

At first glance, this notion looks to be far from geometry.  However, the~next statement shows its interactions with the~real topology.

\begin{thm}[Mikhalkin-Rullg\aa rd \cite{MiRu}]
\label{thm:Harnack}
Let the~Newton polygon $\Delta_P$ of a polynomial~$P(z,w)$ have a non-zero area. Then the~following three conditions are equivalent:
\begin{enumerate}
\item The amoeba $\mathscr{A}_P$ has maximal area. 
\item There are constants $a,b,c\in \mathbb{C}^{\times}$ such that $a P(bz,cw)$  has real coefficients.
			The logarithmic mapping $\Log: V\to\mathbb{R}^2$ is at most two-to-one, where $V=P^{-1}(0)\cap (\mathbb{C}^{\times})^2$.
\item There are constants $a,b,c\in \mathbb{C}^{\times}$ such that $a P(bz,cw)$  has real coefficients. The corresponding real algebraic curve is a Harnack curve for the~polygon~$\Delta_P$.
\end{enumerate}
\end{thm}

Point out the~properties of amoebas of Harnack polynomials that are important in our study.

\begin{lem}\label{lem:1}
	Given a~Harnack polynomial $P(z,w)$ with real coefficients, let $E_\nu$ be a~component of the~complement $\mathbb{R}^2\setminus\mathscr{A}_V$.
	Then the~image $\Arg\circ\Log^{-1}(\partial E_\nu)$ consists of a~single point~$(\varphi_\nu,\psi_\nu)$ from the~set
	\begin{equation*}
		\Theta = \{(0,0),\,(0,\pi),\,(\pi,0),\,(\pi,\pi)\}.
	\end{equation*}
\end{lem}

\begin{proof}
Consider a~Harnack polynomial $P(z,w)$ and the~complex curve~$V$ in 
$(\mathbb{C}^{\times})^2$ defined by~$P$. The boundary~$\partial \mathscr{A}_V$
of its amoeba consists of fold critical points of the~projection
$\left.\textup{Log}\right|_V: V\to \mathscr{A}_V$. For each point on the~boundary $\partial \mathscr{A}_V$, its preimage by $\left.\textup{Log}\right|_V$ is a point, while for a point in the interior of~$\mathscr{A}_V$ the~preimage consists of two points.

Now suppose that $(x_0,y_0)$ is a point on the~boundary~$\partial E_\nu$ of a~component~$E_\nu$. Then the~real torus $\textup{Log}^{-1}(x_0,y_0)$ intersects the~curve~$V$ in a~unique point $(z_0,w_0)$.
\begin{figure}[h]
\label{fig:2}
\centering
\includegraphics[scale=.75]{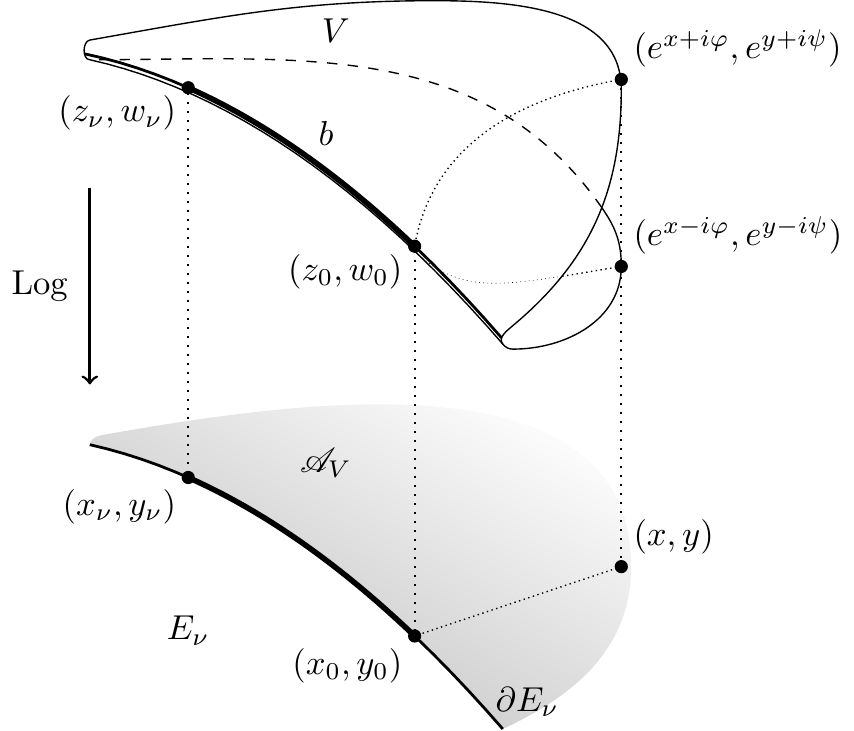}
\caption{\small  Illustration to the~proof of Lemma~\ref{lem:1}.}
\end{figure}

The polynomial~$P$ has real coefficients by the hypothesis, so that
the~complex conjugate $(\overline{z}_0,\overline{w}_0)$ of the point $(z_0,w_0)$ lies on $V$ also. 
The~logarithmic projection~$\left.\textup{Log}\right|_V: V\to \mathscr{A}_V$ maps the~conjugate to the~same point $(x_0,y_0)$ on the~boundary~$\partial E_\nu$. Thus, the points $(z_0,w_0)$ and $(\overline{z_0},\overline{w_0})$ coincide. 
The point $(z_0,w_0)$ is real, and $\Arg\circ\Log^{-1}(x_0,y_0)$ is a~point~$(\varphi_0, \psi_0)$ in $\Theta$.

Assume that $(x_\nu,y_\nu)=\Log(z_\nu,w_\nu)$ is a point on $\partial E_\nu$ such that the~point $\Arg\circ\Log^{-1}(x_\nu,y_\nu)$ belongs to $\Theta\setminus\{(\varphi_0,\psi_0)\}$.
We shall show that this leads to a contradiction. 

Consider a continuous path $b : [0;1]\to V$ from $b(0)=(z_0,w_0)$ to $b(1)=(z_\nu,w_\nu)$ such that $b([0;1])=V\cap\Log^{-1}[(x_0,y_0);(x_\nu,y_\nu)]$, where $[(x_0,y_0);(x_\nu,y_\nu)]$ is an arc on $\partial E_\nu$ bounded by $(x_0,y_0)$ and $(x_\nu,y_\nu)$.
At least one of functions $\Re z \circ b$, $\Re w \circ b$ has values of different signs at $t=0$ and $t=1$.
To be definite, assume that $(\Re z \circ b)(0)\cdot(\Re z \circ b)(1)$ is negative.
So there is $\tilde{t}_0\in (0;1)$ with $\Re z (b(\tilde{t}_0))=0$, i.e. a point $(x_1,y_1)$
which lifts to the~point $(0,w_1)=b(t_0)$ on $V$. However, $V$ is defined as a zero locus of~$P$ in $(\mathbb{C}^{\times})^2$.
This contradiction completes the~proof.
\end{proof}

The four marked points $(\varphi_\nu, \psi_\nu)$ on Fig.~1 exhaust the family~$\Theta$ for the~Harnack polynomial $P(z,w)=z^2w-4zw+zw^2+1$.

When $P$ is a normalized Harnack polynomial, that is a~Harnack polynomial with real coefficients and some special condition (see~\cite{Pa} for a definition), M.~Passare in~\cite{Pa} gave an explicit formula
for amoeba-to-coamoeba mapping
\begin{equation*}
	\Arg\circ\Log^{-1}(x,y)
	=\left(\pm\pi\frac{\partial\mathcal{N}_P}{\partial y}(x,y), \mp\pi\frac{\partial\mathcal{N}_P}{\partial x}(x,y)\right),
\end{equation*}
which proves Lemma~\ref{lem:1} in the~corresponding case.

In general, one has
\begin{lem}\label{lem:2}
	Given a Harnack polynomial $P(z,w)$, let $E_\nu$ be a component of the~complement $\RR^n\setminus\mathscr{A}_V$.
	Then the~image $\Arg\circ\Log^{-1}(\partial E_\nu)$ is a single point~$(\varphi_\nu,\psi_\nu)$.
\end{lem}

\begin{proof}
By Theorem~\ref{thm:Harnack} there exist $a,b,c\in\CC^\times$ such that $\widetilde{P}(z,w)=aP(bz,cw)$ has real coefficients.
Applying Lemma~\ref{lem:1} to $\widetilde{P}(z,w)$ one gets that the~image of $\partial \widetilde{E}_\nu$, where $\widetilde{E}_\nu$ is the~component of $\RR^2\setminus\mathscr{A}_{\widetilde{P}}$, by the~map $\Arg\circ\Log^{-1}$ is a point from $\Theta$.

Multiplication by a non-zero constant $a$ does not affect the~zero locus $\widetilde{V}$ of the~polynomial $\widetilde{P}$.
The linear transformation $(z,w)\mapsto(bz,cw)$ induces a translation of $\mathscr{A}_{\widetilde{P}}$ by a vector $(\log|b|,\log|c|)$ and a translation of $\mathscr{A}'_{\widetilde{P}}$ by $(\arg(b),\arg(c))$.
So $\Arg\circ\Log^{-1}(\partial E_\nu)$ is a point in $\Theta+(\arg(b),\arg(c))$.
\end{proof}

\section{Proof of the main result}
\label{s:3}

Let $\overline{\mathbb{C}^2}=\mathbb{C}^2\cup\{\infty\}$, and $L=L_1\cup L_2$ be the~union of two lines $\{z=0\}$ and $\{w=0\}$ in $\mathbb{C}^2=\mathbb{C}_z\times\mathbb{C}_w$. Since $\overline{\mathbb{C}^2}$ is homeomorphic to the~sphere $S^4$, refer to $\overline{\mathbb{C}^2}$ as 
the~spherical compactification of $\mathbb{C}^2$. Futher,  we denote by $\overline{X}$ the~closure in~$\overline{\mathbb{C}^2}$ of $X\subset\mathbb{C}^2$.

Since $(\mathbb{C}^{\times})^2\setminus V = \mathbb{C}^2\setminus (V\cup L)= \overline{\mathbb{C}^2}\setminus\overline{V\cup L},$
one has 
\begin{equation*}
H_2((\mathbb{C}^{\times})^2\setminus V) = H_2(\overline{\mathbb{C}^2}\setminus\overline{V\cup L}).
\end{equation*}
Next, by the~Alexander-Pontryagin duality~\cite{Pontr},
$H_2(\overline{\mathbb{C}^2}\setminus\overline{V\cup L})\cong H_1(\overline{V\cup L}).$
Thus, 
$$
H_2((\mathbb{C}^{\times})^2\setminus V)\cong H_1(\overline{V\cup L}).
$$
Note that $\overline{V\cup L}$ is homeomorphic to the~topological sum of three Riemann 
surfaces defined by $P(z,w)=0$, $z=0$ and $w=0$ with certain points identified (see Fig.~3). 

\begin{figure}[ht]
\label{fig:3}
\centering
\includegraphics[scale=.8]{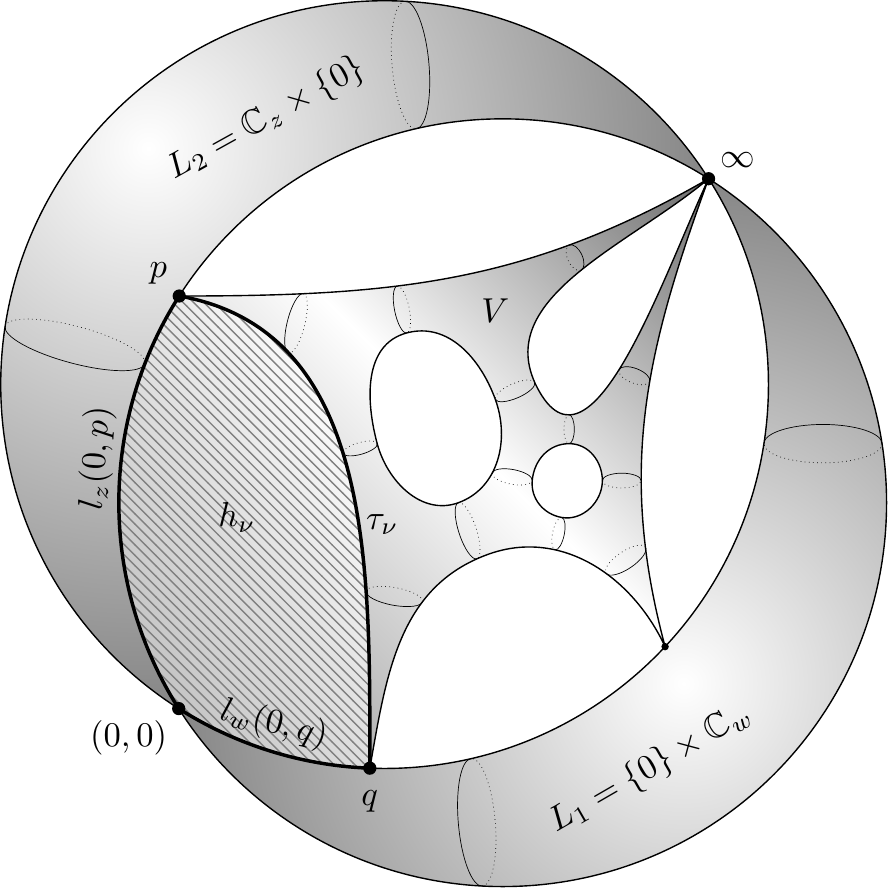}
\caption{\small  Illustration to the~proof of the~main theorem.}
\end{figure}

Recall that the~family $\{\Gamma_\nu\}$ consists of toric cycles in $(\mathbb{C}^{\times})^2\setminus V$ associated with components~$E_\nu$ of the~complement $\mathbb{R}^2\setminus \mathscr{A}_V$. We are going to construct a~family of 1-cycles~$\sigma_\nu$ in $\overline{V\cup L}$
dual to the~family $\{\Gamma_\nu\}$ in the folowing sense
\begin{equation}
\label{eq:link}
	\textup{link}(\sigma_\mu, \Gamma_\nu)=\pm\delta_{\mu\nu}=
	\begin{cases}
   		0,\,& \text{if $\mu\neq\nu$}, \\
   		\pm1,\,& \text{if $\mu=\nu$}.
  	\end{cases}
\end{equation}
Existence of  the family~$\{\sigma_\nu\}$ with such property  of pairing implies, obviously, homological
independence of the~families $\{\Gamma_\nu\}$ and $\{\sigma_\nu\}$.

By Lemma~\ref{lem:2} we know that the~image~$\Arg\circ\Log^{-1}(\partial E_\nu)$  is a single point $(\varphi_\nu, \psi_\nu)$. Therefore we can define the lifting of $\partial E_\nu$ to the curve~$V$ as 
\begin{equation*}
	\tau_\nu=\{(e^{x+i \varphi_\nu}, e^{y+i \psi_\nu}): (x,y)\in \partial E_\nu\}\subset (\mathbb{C}^{\times})^2,
\end{equation*}
that is $\tau_\nu=V\cap\Log^{-1}(\partial E_\nu)$. Now, we need to construct compact cycles $\sigma_\nu$
 in $\overline{V\cup L}$ using $\tau_\nu$.

Consider the boundary $\partial \tau_\nu =\overline{\tau}_\nu \setminus (\mathbb{C}^{\times})^2$ of the lifting
 in $\overline{\mathbb{C}^2}$. Its cardinality may be $0$, $1,$ or $2$. For instance, when $E_{\nu}$  is a bounded
 component, $\partial \tau_\nu$ is empty. In this case, we put $\sigma_{\nu}=\tau_\nu$.

Let $E_\nu$ be an unbounded component. The set $\partial \tau_\nu$ consists of one or two points 
in the ray $\overline{\{t(e^{i\varphi_\nu},0): t\geq 0\}}\subset \overline{L_2}$ or in the ray
$\overline{\{t(0,e^{i\psi_\nu}): t\geq 0\}}\subset \overline{L_1}$.

The particular configuration of points in $\partial \tau_\nu$ depends on the recession cone $C(E_\nu)$ 
of the component~$E_\nu$. Since~$C(E_\nu)$ is a~sector (possibly degenerated to a~ray), its position in $\mathbb{R}^2$ can be described by the~pair $(u,v)\in T^2=S^1\times S^1$, where $S^1$ is the~unit circle. In other words,  all the~possible shapes of $E_\nu$ can be identified with a~subset on the~torus $T^2$. (In some sense, $T^2$ is the~configurational space of components $E_\nu$).

\begin{figure}[h]
\label{fig:4}
\centering
\includegraphics[scale=.63]{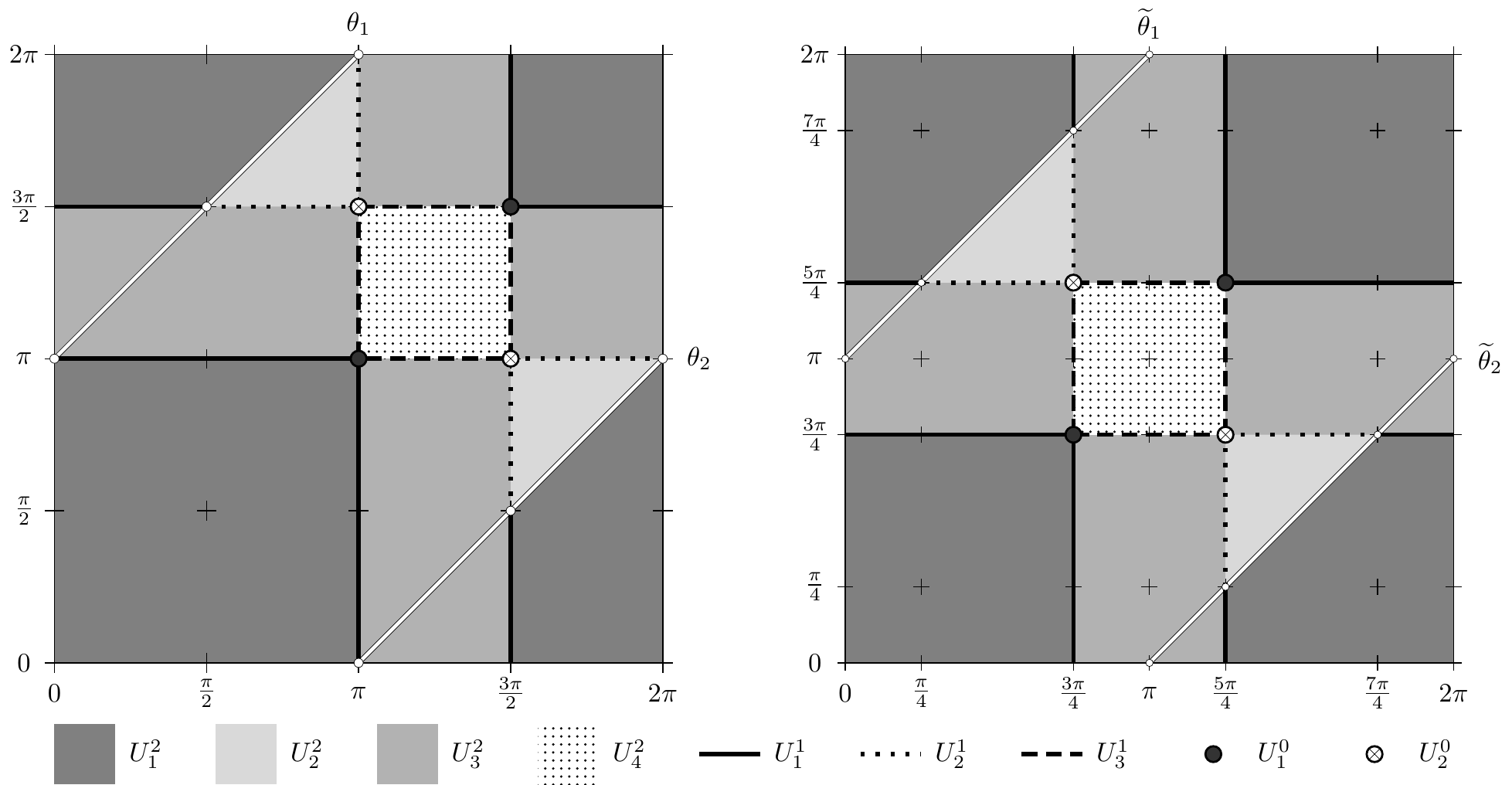}
\caption{\small  Partition~(\ref{eq:part}) of the~torus $T^2$. 
Components $U_j^i$ are given in local coordinates $0\leq\theta_1,\theta_2< 2\pi$ 
(left; the~angles $\theta_1$, $\theta_2$ are measured from the~horizontal coordinate axis in $\mathbb{R}^2$ to the~vectors~$u, v$  counter-clockwise) and $0\leq\tilde{\theta}_1,\tilde{\theta}_2<2\pi$ 
(right; the~angles $\tilde{\theta}_1$, $\tilde{\theta}_2$ are measured from the~diagonal of the~positive quadrant in $\mathbb{R}^2$ to the~vectors~$u, v$  counter-clockwise). Cuts $|\theta_2-\theta_1|=\pi$
(or $|\tilde{\theta}_2-\tilde{\theta}_1|=\pi$) corresponds to the~degeneration of the~Newton polytope $\Delta_P$.
}
\end{figure}

In order to describe constructions of $\sigma_\nu$ for all the~possible shapes of~$E_\nu$,
we consider the~partition of the~torus
\begin{equation}
\label{eq:part}
	T^2=	U_1^0\sqcup U_2^0\sqcup U_1^1\sqcup U_2^1\sqcup U_3^1\sqcup 
			U_1^2\sqcup U_2^2\sqcup U_3^2\sqcup U_4^2, 
\end{equation}
where components~$U_j^i$ are defined as in Fig.~4. Table~1 establishes the~correspondence between components~$U_j^i$ of the~partition, possible sets $\partial \tau_\nu$ and constructions of the~cycle~$\sigma_\nu$.
The construction of $\sigma_\nu$ involves the lifting~$\tau_\nu$ and segments $l_z(a,b)$, $l_w(a,b)$  in the~closure of the rays $\{t(e^{i\varphi_\nu},0): t\geq 0\}\subset \overline{L_2}$, $\{t(0,e^{i\psi_\nu}): t\geq 0\}\subset \overline{L_1}$. The~segments $l_z(a,b)$, $l_w(a,b)$ join points (possibly infinite)
$a,b$ on the corresponding rays. We denote by $p_i$ points in $\overline{V}\cap L_2$, by $q_j$ points
in $\overline{V}\cap L_1$ and by $O$ the origin $(0,0)$.

 For example, $U_2^0$ consists of two points, each encoding the~same recession cone
 generated by $(-1,0)$ and $(0,-1)$. Thus, $\partial \tau_\nu$ consists of two points $p\in \overline{V}\cap L_2$ and $q\in \overline{V}\cap L_1$, and to construct $\sigma_\nu$ one need to add $l_z(O,p)$ and $l_w(O,q)$ to $\tau_\nu$.  (See Fig.~3).

\begin{table}[tp]%
\caption{The correspondence between components~$U_j^i$ of the~partition~(\ref{eq:part}),
possible sets $\partial \tau_\nu$ and constructions of the~cycle~$\sigma_\nu$.}
\centering %
\begin{tabular}{ccccc}
\toprule %
&component of $T^2$ 		& 	$\partial \tau_\nu$			&construction of $\sigma_\nu$									&	\\\midrule%
	
&$U_1^2$	& 	$\{\infty\}$		&	$\tau_\nu\cup\{\infty\}$ 			&    \\
&			&						&													&	\\\midrule%

&$U_2^2$	& 	$\{\infty\}$		&	$\tau_\nu\cup\overline{(l_z(O,\infty)\cup l_w(O,\infty))}$ &\\
&			&						&														&\\\midrule%

&$U_3^2$	&	$\{O,\infty\}$	&	either $\tau_\nu\cup\overline{l_z(O,\infty)}$	&\\
&			&						&	or $\tau_\nu\cup\overline{l_w(O,\infty)}$	&\\\midrule%

&$U_4^2$	&	$\{O\}$			&	$\tau_\nu\cup\{O\}$					&\\
&			&						&														&\\\midrule%

&$U_1^1$	&   either $\{p,\infty\}$		&	$\tau_\nu\cup\overline{l_z(p,\infty)}$	&\\
&			&	or $\{q,\infty\}$		&	$\tau_\nu\cup\overline{l_w(q,\infty)}$	&\\\midrule

&$U_2^1$	&	either $\{p,\infty\}$		&	$\tau_\nu\cup l_z(O,p)\cup\overline{l_w(O,\infty)}$	&\\
&			&	or $\{q,\infty\}$		&	$\tau_\nu\cup l_w(O,q)\cup\overline{l_z(O,\infty)}$	&\\\midrule

&$U_3^1$	&	either $\{O,p\}$		&	$\tau_\nu\cup l_z(O,p)$	&\\
&			&	or $\{O,q\}$		&	$\tau_\nu\cup l_w(O,q)$	&\\\midrule

&$U_1^0$	&	either $\{p_1,p_2\}$		&	$\tau_\nu\cup l_z(p_1,p_2)$	&\\
&			&	or $\{q_1,q_2\}$		&	$\tau_\nu\cup l_w(q_1,q_2)$	&\\\midrule

&$U_2^0$	&	$\{p,q\}$			&	$\tau_\nu\cup l_z(O,p)\cup l_w(O,q)$	 &\\
&			&						&													 &\\\bottomrule
\end{tabular}
\end{table}

Let us write the~left-hand side of~(\ref{eq:link}) 
\begin{equation*}
	\textup{link}(\sigma_\mu, \Gamma_\nu)=\textup{ind}(h_\mu, \Gamma_\nu)
\end{equation*}
as the~intersection index in $\overline{\mathbb{C}^2}$ of the~cycle $\Gamma_\nu$ and the~$2$-dimensional chain
\begin{equation*}
	h_\mu=\{(e^{x+i \varphi_\mu}, e^{y+i \psi_\mu}): (x,y)\in  E_\mu\}\cup \sigma_\mu, 
\end{equation*}
such that $\partial h_\mu=\sigma_\mu.$ The toric cycle $\Gamma_\nu$ has the~form $\Log^{-1}(x,y)$, where $(x,y)\in E_\nu$. So the~chain $h_\mu$ does not intersect $\Gamma_\nu$ if $\mu\neq \nu$. Meanwhile, if $\mu=\nu$ the~cycle $\Gamma_\nu$ intersects $h_\mu$ in a~single point $(e^{x+i\varphi_\nu},e^{y+i\psi_\nu})$. Therefore, 
$\textup{ind}(h_\nu, \alpha(\Gamma_\mu))=\pm\delta_{\nu\mu}$. Q.E.D.

\subsection*{Acknowledgments} 
 The first author was financed by the~grant of the~President of the~Russian Federation for state support of leading scientific schools NSh-9149.2016.1. 
 The second author was supported by the~grant of the~Russian Federation Government for research under supervision of leading scientist at Siberian Federal University, contract №14.Y26.31.0006.

\end{document}